\documentclass[a4paper,11pt]{amsart}

\usepackage{amsmath, amsthm, amsfonts, amssymb}
\usepackage[bookmarks=false]{hyperref}
\usepackage{enumerate}

\usepackage{color}

\numberwithin{equation}{section}

\newtheorem*{mainthm}{Main theorem}

\newtheorem{thm}{Theorem}[section]
\newtheorem{lem}[thm]{Lemma}

\newtheorem{cor}[thm]{Corollary}
\newtheorem{prop}[thm]{Proposition}

\theoremstyle{definition}
\newtheorem{rem}[thm]{Remark}

\newtheorem{df}[thm]{Definition}

\newtheorem*{definition}{Definition}

\newcommand{\R}{\mathbf{R}}
\newcommand{\C}{\mathbf{C}}

\newcommand{\F}{\mathbf{F}}

\newcommand{\N}{\mathbf{N}}

\newcommand{\cF}{\mathcal{F}}
\newcommand{\cZ}{\mathcal{Z}}
\newcommand{\cU}{\mathcal{U}}

\newcommand{\cN}{\mathcal{N}}
\newcommand{\cQ}{\mathcal{Q}}
\newcommand{\QN}{\mathcal{QN}}

\newcommand{\Ad}{\operatorname{Ad}}
\newcommand{\id}{\text{\rm id}}

\newcommand{\Aut}{\operatorname{Aut}}
\newcommand{\rL}{\mathord{\text{\rm L}}}
\newcommand{\rB}{\mathord{\text{\rm B}}}

\newcommand{\rE}{\mathord{\text{\rm E}}}

\newcommand{\core}{\mathord{\text{\rm c}}}

\newcommand{\cb}{\mathord{\text{\rm cb}}}

\newcommand{\Tr}{\mathord{\text{\rm Tr}}}

\newcommand{\ot}{\otimes}
\newcommand{\ovt}{\mathbin{\overline{\otimes}}}
\newcommand{\otalg}{\otimes_{\text{\rm alg}}}
\newcommand{\op}{^\text{\rm op}}

\newcommand{\Mtil}{\widetilde{M}}
\newcommand{\al}{\alpha}
\newcommand{\dpr}{^{\prime\prime}}
\newcommand{\actson}{\curvearrowright}
\newcommand{\sN}{s\mathcal{N}}

\newcommand{\Qtil}{\widetilde{Q}}
\newcommand{\cV}{\mathcal{V}}
\newcommand{\cRtil}{\widetilde{\mathcal{R}}}
\newcommand{\recht}{\rightarrow}
\newcommand{\cR}{\mathcal{R}}
\newcommand{\vphi}{\varphi}
\newcommand{\Xtil}{\widetilde{X}}
\newcommand{\Ptil}{\widetilde{P}}
\newcommand{\abar}{\overline{a}}
\newcommand{\ubar}{\overline{u}}
\newcommand{\xbar}{\overline{x}}
\newcommand{\vbar}{\overline{v}}
\newcommand{\om}{\omega}
\newcommand{\be}{\beta}
\newcommand{\eps}{\varepsilon}
\newcommand{\real}{\operatorname{Re}}
\newcommand{\cS}{\mathcal{S}}
\newcommand{\Om}{\Omega}
\newcommand{\SO}{\operatorname{SO}}
\newcommand{\SU}{\operatorname{SU}}
\newcommand{\cG}{\mathcal{G}}

\begin{document}

\title{Strong solidity of free Araki--Woods factors}

\begin{abstract}
We show that Shlyakhtenko's free Araki--Woods factors are strongly solid, meaning that for any diffuse amenable von Neumann subalgebra that is the range of a normal conditional expectation, the normalizer remains amenable. This provides the first class of nonamenable strongly solid type ${\rm III}$ factors.
\end{abstract}

\author{R\'emi Boutonnet}
\address{Institut de Math\'ematiques de Bordeaux \\ CNRS \\ Universit\'e Bordeaux I \\ 33405 Talence \\ FRANCE}
\email{remi.boutonnet@math.u-bordeaux.fr}
\thanks{RB is supported by NSF Career Grant DMS 1253402}

\author{Cyril Houdayer}
\address{Laboratoire de Math\'ematiques d'Orsay \\ Universit\'e Paris-Sud \\ CNRS \\ Universit\'e Paris-Saclay \\ 91405 Orsay \\ FRANCE}
\email{cyril.houdayer@math.u-psud.fr}
\thanks{CH is supported by ERC Starting Grant GAN 637601}

\author{Stefaan Vaes}
\address{Department of Mathematics\\KU Leuven\\Celestijnenlaan 200B\\B-3001 Leuven\\BELGIUM}
\email{stefaan.vaes@kuleuven.be}
\thanks{SV is supported by ERC Consolidator Grant 614195, and by long term structural funding~-- Methusalem grant of the Flemish Government}

\subjclass[2010]{46L10, 46L54, 46L36}
\keywords{Free group factors; Free Araki--Woods factors; Popa's deformation/rigidity theory; Strong solidity; Type ${\rm III}$ factors}

\maketitle

\section{Introduction}

A von Neumann algebra $N$ is {\em amenable} if there exists a norm one projection $\Phi : \rB(\rL^2(N)) \to N$. By Connes's fundamental result \cite{Co75}, any amenable von Neumann algebra is approximately finite dimensional. Moreover, the class of amenable factors with separable predual is completely classified by the flow of weights \cite{Co72, Co75, Ha85, Kr75}. In particular, there exists a unique amenable ${\rm II_1}$ factor with separable predual: it is the hyperfinite ${\rm II_1}$ factor $R$ of Murray and von Neumann \cite{MvN43}.

Starting with \cite{Po01}, Popa's deformation/rigidity theory has lead to far reaching classification and structure theorems for nonamenable factors. Particular attention was given to several types of indecomposability results for von Neumann algebras $M$, like primeness (the impossibility to write a factor as a nontrivial tensor product), solidity (inside $M$, there is no room for a nonamenable subalgebra and a diffuse subalgebra to commute) and absence of Cartan subalgebras (the impossibility to write a factor as one coming from a group action or an equivalence relation). The strongest possible indecomposability property for a von Neumann algebra $M$, encompassing primeness, solidity and the absence of Cartan subalgebras was discovered in Ozawa and Popa's breakthrough article \cite{OP07} and called \emph{strong solidity.}

\begin{definition}
Let $M$ be any diffuse von Neumann algebra. Following \cite{OP07}, we say that $M$ is {\em strongly solid} if for any diffuse amenable von Neumann subalgebra $Q \subset M$ with faithful normal conditional expectation $\rE_Q : M \to Q$, the normalizer $\mathcal N_M(Q)\dpr$ generated by $\mathcal N_M(Q) := \{u \in \mathcal U(M) \mid uQu^* = Q\}$ remains amenable.
\end{definition}

In \cite{OP07}, the free group factors $\rL(\F_n)$, $2 \leq n \leq +\infty$, were shown to be strongly solid. This result strengthens both Voiculescu's \cite{Vo95} proving that the free group factors have no Cartan subalgebra and Ozawa's \cite{Oz03} proving that the free group factors are solid.

The type ${\rm III}$ counterparts of the free group factors are Shlyakhtenko's free Araki--Woods factors \cite{Sh96}, defined via Voiculescu's free Gaussian functor \cite{Vo85, VDN92}. Although free Araki--Woods factors were shown to be solid and to have no Cartan subalgebra \cite{HR10}, strong solidity remained an open problem. So far, there were even no examples of strongly solid type ${\rm III}$ factors altogether. As we explain in detail below, the strong solidity of a type ${\rm III}$ factor $M$ is closely related to a relative strong solidity property of its continuous core $\core(M)$, which is a semifinite von Neumann algebra. The main results of \cite{OP07} apply to the finite corners $p \core(M) p$. But, in order to be applicable to $\core(M)$, we need to control, inside $p \core(M) p$, not only normalizers but also so-called groupoid-normalizers or stable normalizers. That is precisely the problem that we solve in the present paper and that allows us to prove that all free Araki--Woods factors are strongly solid.

Following \cite{Sh96}, to any orthogonal representation $U : \R \curvearrowright H_\R$ on a real Hilbert space, one associates the free Araki--Woods von Neumann algebra $\Gamma(H_\R, U)\dpr$, which comes equipped with the {\it free quasi-free state} $\varphi_U$ (see Section \ref{preliminaries} for a detailed construction).

Free Araki--Woods factors were first studied in the framework of Voiculescu's free probability theory. A complete description of their type classification as well as fullness and computation of Connes's Sd and $\tau$ invariants was obtained in \cite{Sh96, Sh97a, Sh97b, Sh02} (see also the survey \cite{Va04}). We have $\Gamma(H_\R, \id)\dpr \cong \rL(\F_{\dim (H_\R)})$ when $U = 1_{H_\R}$ and $\Gamma(H_\R, U)\dpr$ is a full type ${\rm III}$ factor when $U \neq 1_{H_\R}$. Moreover, the free Araki--Woods factor $\Gamma(H_\R, U)\dpr$ admits a discrete decomposition in the sense of \cite{Co74} if and only if the orthogonal representation $U : \R \curvearrowright H_\R$ is almost periodic. The class of free Araki--Woods factors is quite large. Indeed, there are uncountably many pairwise nonisomorphic type ${\rm III_1}$ free Araki--Woods factors that admit a discrete decomposition \cite{Sh96} as well as uncountably many that do not \cite{Sh02}.
More recently, free Araki--Woods factors were studied using Popa's deformation/rigidity theory. This new approach allowed to obtain various indecomposability results in \cite{Ho08, HR14} and to show that free Araki--Woods factors satisfy the complete metric approximation property (CMAP) \cite[Theorem A]{HR10} and have no Cartan subalgebra \cite[Theorem B]{HR10}.

The following is then our main result.

\begin{mainthm}
For every orthogonal representation $U : \R \curvearrowright H_\R$ such that $\dim H_\R \geq 2$, the free Araki--Woods factor $\Gamma(H_\R, U)\dpr$ is strongly solid.
\end{mainthm}

The main step to prove this result is to adapt the proof of Ozawa--Popa's \cite[Theorem 3.5]{OP07} so as to cover as well the \emph{groupoid-normalizer} or \emph{stable normalizer} of $Q \subset M$, defined as the von Neumann algebra generated by $\{x \in M \mid x Q x^* \subset Q \;\;\text{and}\;\; x^* Q x \subset Q \}$. We thus prove in particular that for any diffuse amenable $Q \subset \rL(\F_n)$, the stable normalizer of $Q$ remains amenable.

To prove that a free Araki--Woods factor $M = \Gamma(H_\R, U)\dpr$ is strongly solid, we proceed as follows. Fix a diffuse amenable von Neumann subalgebra $Q \subset M$ with expectation (meaning that there exists a faithful normal conditional expectation $\rE_Q : M \to Q$). We have to prove that $P := \cN_M(Q)\dpr$ remains amenable.

Using Connes's continuous decomposition \cite{Co72}, we have natural inclusions of the semifinite continuous cores $\core(Q) \subset \core(P) \subset \core(M)$. In general, it is not true that $\core(P)$ is contained in the normalizer of $\core(Q)$. However, since $M$ is solid (see e.g.\ \cite[Theorem A]{HR14}), we may replace $Q$ by $Q \vee (Q' \cap M)$ and assume that $Q' \cap M = \cZ(Q)$. Then, $\core(P)$ is contained in the normalizer of $\core(Q)$ (see Lemma \ref{lem-modular}) and by Takesaki's duality theorem \cite[Theorem X.2.3]{Ta03}, it suffices to show that $\mathcal N_{\core(M)}(\core(Q))\dpr$ is amenable.

Cutting down by any nonzero finite projection in $\core(Q)$, we obtain the inclusions of finite (tracial) von Neumann algebras $\mathcal Q \subset \mathcal P \subset \mathcal M$ where $\mathcal Q = p \core(Q) p$, $\mathcal P = p (\mathcal N_{\core(M)}(\core(Q))\dpr) p$ and $\mathcal M = p \core(M) p$. It is important to point out that $\mathcal P$ need not be contained in the normalizer of $\mathcal Q$, but is always contained in the stable normalizer of $\mathcal Q$.

By \cite[Theorem A]{HR10}, the tracial von Neumann algebra $\mathcal M$ has CMAP and it has a natural malleable deformation in the sense of \cite{Po03}. So we are exactly in the setting of \cite{OP07}, except that we need to extend their main result on weak compactness to stable normalizers. We do this in Proposition \ref{prop.WC}.

\subsection*{Acknowlegment} R.B.\ is grateful to Adrian Ioana for many stimulating discussions on this project.

\section{Preliminaries}\label{preliminaries}

\subsection{Background on $\sigma$-finite von Neumann algebras}

For any von Neumann algebra $M$, we denote by $\mathcal Z(M)$ the center of $M$, by $\mathcal U(M)$ the group of unitaries in $M$ and by $(M, \rL^2(M), J, \rL^2(M)^+)$ the standard form of $M$. We say that an inclusion of von Neumann algebras $P \subset M$ is {\em with expectation} if there exists a faithful normal conditional expectation $\rE_P : M \to P$. We say that a $\sigma$-finite von Neumann algebra $M$ is {\em tracial} if it is endowed with a faithful normal tracial state $\tau$.

Let $M$ be any $\sigma$-finite von Neumann algebra with predual $M_\ast$ and $\varphi \in M_\ast$ any faithful state. We denote by $\sigma^\varphi$ the modular automorphism group of the state  $\varphi$.  The {\em continuous core} of $M$ with respect to $\varphi$, denoted by $\core_\varphi(M)$, is the crossed product von Neumann algebra $M \rtimes_{\sigma^\varphi} \R$.  The natural inclusion $\pi_\varphi: M \to \core_\varphi(M)$ and the unitary representation $\lambda_\varphi: \R \to \core_\varphi(M)$ satisfy the {\em covariance} relation
\[\lambda_\varphi(t) \pi_\varphi(x) \lambda_\varphi(t)^* = \pi_\varphi(\sigma^\varphi_t(x)) \quad \text{ for all } x \in M \text{ and all } t \in \R.\]

Put $\rL_\varphi (\R) = \lambda_\varphi(\R)\dpr$. There is a unique faithful normal conditional expectation $\rE_{\rL_\varphi (\R)}: \core_{\varphi}(M) \to \rL_\varphi(\R)$ satisfying $\rE_{\rL_\varphi (\R)}(\pi_\varphi(x) \lambda_\varphi(t)) = \varphi(x) \lambda_\varphi(t)$ for all $x \in M$ and all $t \in \R$. The faithful normal semifinite (fns) weight defined by $f \mapsto \int_{\R} \exp(-s)f(s) \, {\rm d}s$ on $\rL^\infty(\R)$ gives rise to a fns weight $\Tr_\varphi$ on $\rL_\varphi(\R)$ via the Fourier transform. The formula $\Tr_\varphi = \Tr_\varphi \circ \rE_{\rL_\varphi (\R)}$ extends it to a fns trace on $\core_\varphi(M)$.

Because of Connes's Radon--Nikodym cocycle theorem \cite[Th\'eor\`eme 1.2.1]{Co72} (see also \cite[Theorem VIII.3.3]{Ta03}), the semifinite von Neumann algebra $\core_\varphi(M)$ together with its trace $\Tr_\varphi$ does not depend on the choice of $\varphi$ in the following precise sense. If $\psi \in M_\ast$ is another faithful state, there is a canonical surjective $\ast$-isomorphism
$\Pi_{\varphi,\psi} : \core_\psi(M) \to \core_{\varphi}(M)$ such that $\Pi_{\varphi,\psi} \circ \pi_\psi = \pi_\varphi$ and $\Tr_\varphi \circ \Pi_{\varphi,\psi} = \Tr_\psi$. Note however that $\Pi_{\varphi,\psi}$ does not map the subalgebra $\rL_\psi(\R) \subset \core_\psi(M)$ onto the subalgebra $\rL_\varphi(\R) \subset \core_\varphi(M)$ (and hence we use the symbol $\rL_\varphi(\R)$ instead of the usual $\rL(\R)$).

\subsection{Free Araki--Woods factors}

Let $H_{\R}$ be any real Hilbert space and $U : \R \curvearrowright H_\R$ any orthogonal representation. Denote by $H = H_{\R} \otimes_{\R} \C = H_\R \oplus {\rm i} H_\R$ the complexified Hilbert space, by $I : H \to H : \xi + {\rm i} \eta \mapsto \xi - {\rm i} \eta$ the canonical anti-unitary involution on $H$ and by $A$ the infinitesimal generator of $U : \R \curvearrowright H$, that is, $U_t = A^{{\rm i}t}$ for all $t \in \R$. Observe that $j : H_{\R} \to H :\zeta \mapsto (\frac{2}{A^{-1} + 1})^{1/2}\zeta$ defines an isometric embedding of $H_{\R}$ into $H$. Moreover, we have $IAI = A^{-1}$. Put $K_{\R} := j(H_{\R})$. It is easy to see that $K_\R \cap {\rm i} K_\R = \{0\}$ and that $K_\R + {\rm i} K_\R$ is dense in $H$.

We introduce the \emph{full Fock space} of $H$:
\begin{equation*}
\mathcal{F}(H) =\C\Omega \oplus \bigoplus_{n = 1}^{\infty} H^{\otimes n}.
\end{equation*}
The unit vector $\Omega$ is called the \emph{vacuum vector}. For all $\xi \in H$, define the {\it left creation} operator $\ell(\xi) : \mathcal{F}(H) \to \mathcal{F}(H)$ by
\begin{equation*}
\left\{
{\begin{array}{l} \ell(\xi)\Omega = \xi, \\
\ell(\xi)(\xi_1 \otimes \cdots \otimes \xi_n) = \xi \otimes \xi_1 \otimes \cdots \otimes \xi_n.
\end{array}} \right.
\end{equation*}
We have $\|\ell(\xi)\|_\infty = \|\xi\|$ and $\ell(\xi)$ is an isometry if $\|\xi\| = 1$. For all $\xi \in K_\R$, put $W(\xi) := \ell(\xi) + \ell(\xi)^*$. The crucial result of Voiculescu \cite[Lemma 2.6.3]{VDN92} is that the distribution of the self-adjoint operator $W(\xi)$ with respect to the vector state $\varphi_U = \langle \, \cdot \,\Omega, \Omega\rangle$ is the semicircular law of Wigner supported on the interval $[-2\|\xi\|, 2\|\xi\|]$.

\begin{df}[Shlyakhtenko, \cite{Sh96}]
Let $U : \R \curvearrowright H_\R$ be any orthogonal representation of $\R$ on a real Hilbert space $H_{\R}$. The \emph{free Araki--Woods} von Neumann algebra associated with $U : \R \curvearrowright H_\R$ is defined by $\Gamma(H_\R, U)\dpr := \{W(\xi) : \xi \in K_{\R}\}\dpr$.
\end{df}

The vector state $\varphi_U = \langle \, \cdot \,\Omega, \Omega\rangle$ is called the {\it free quasi-free state} and is faithful on $\Gamma(H_\R, U)\dpr$. The modular automorphism group $\sigma^{\varphi_U}$ satisfies the formula
$$\sigma_t^{\varphi_U}(W(\xi)) = W(U_t \xi) \quad \text{for all } \xi \in K_\R \text{ and all } t \in \R.$$
We also point out that $M = \Gamma(H_\R, U)\dpr$ satisfies Ozawa's condition (AO) (see e.g.\ \cite[Appendix]{HI15}) and hence is {\em solid} by \cite{Oz03, VV05}, that is, for any diffuse subalgebra with expectation $Q \subset M$, the relative commutant $Q' \cap M$ is amenable.

\subsection{Popa's intertwining-by-bimodules}

Popa introduced his method of {\em intertwining-by-bimodules} in \cite{Po01, Po03}. In the present work, we make use of these results in the context of semifinite von Neumann algebras. Let $(M, \Tr)$ be any semifinite $\sigma$-finite von Neumann algebra endowed with a fns trace. Let $p \in M$ be any nonzero finite trace projection and $A \subset pMp$ any von Neumann subalgebra. Let $B \subset M$ be any von Neumann subalgebra such that $\Tr |_B$ is still semifinite and denote by $\rE_B : M \to B$ the unique trace preserving conditional expectation. We say that $A$ {\em embeds into} $B$ {\em inside} $M$ (and write $A \prec_M B$) if there exists a nonzero projection $q \in B$ with $\Tr(q) < +\infty$ such that $A \prec qBq$ in the sense of Popa (inside the finite von Neumann algebra $(p \vee q)M(p \vee q)$). We refer to e.g.\ \cite[Section 2.1]{HR10} for further details.

\section{Stable normalizers in $\rm{II}_1$ factors}\label{stable-normalizer}

In this section, we prove {\it stable strong solidity results} for $\rm{II}_1$ factors. As pointed out in \cite[Proposition 5.2]{Ho09}, strong solidity is preserved under finite amplifications. However as we explain below, there is a priori no reason for strong solidity to be preserved under infinite amplifications. Nevertheless, we prove in this section that for many known cases of strongly solid $\rm{II}_1$ factors, their infinite amplification remains strongly solid.

\begin{df}
Given a von Neumann algebra $M$ and a von Neumann subalgebra $Q \subset M$, define
\[\sN_M(Q) := \{x \in M \mid x Q x^* \subset Q \quad\text{and}\quad x^* Q x \subset Q \} \; .\]
We call the von Neumann algebra $\sN_M(Q)\dpr$ the \emph{stable normalizer} of $Q$ inside $M$.
\end{df}

The terminology \emph{stable} normalizer is motivated by Lemma \ref{lem.prop-PN}(3) saying that the stable normalizer of $Q$ is given by the normalizer of the stabilization $Q \ovt \rB(\ell^2(\N))$ of $Q$.

Note that $xy$ and $x^*$ belong to $\sN_M(Q)$ for all $x,y \in \sN_M(Q)$. Also, $Q \subset \sN_M(Q)$ and the linear span of $\sN_M(Q)$ is a $*$-algebra containing $Q$ and $Q' \cap M$. The polar decomposition $x = v |x|$ of an element $x \in \sN_M(Q)$ satisfies the following properties: $|x| \in Q$, $v$ is a partial isometry whose initial projection $p = v^*v$ and final projection $q= v v^*$ belong to $Q$ and that satisfies $v Q v^* = q Q q$.

Assume moreover that $Q \subset M$ is with expectation. Then $\sN_M(Q)\dpr \subset M$ is also with expectation. Indeed, we may choose a faifhful state $\varphi \in M_\ast$ such that $Q \subset M$ is globally invariant under $\sigma^\varphi$. Then $\sigma_t^\varphi(x) \in \sN_M(Q)$ for all $x \in \sN_M(Q)$ and all $t \in \R$. This implies that $\sN_M(Q)\dpr \subset M$ is globally invariant under $\sigma^\varphi$ and thus $\sN_M(Q)\dpr \subset M$ is with expectation by \cite[Theorem IX.4.2]{Ta03}.

\begin{df}
We say that a diffuse von Neumann algebra $M$ is \emph{stably strongly solid} if for every diffuse amenable von Neumann subalgebra with expectation $A \subset M$, we have that $\sN_M(A)\dpr$ remains amenable.
\end{df}

A priori, strong solidity does not imply stable strong solidity. The reason for this is that the stable normalizer has a qualitatively more general behavior than the normalizer, as can be seen as follows. Assume that $M$ is a II$_1$ factor with tracial state $\tau$. Let $A \subset M$ be a von Neumann subalgebra. When $u \in \cN_M(A)$, then $\Ad u$ defines a trace preserving automorphism of $A$. In particular, the restriction of $\Ad u$ to the center of $A$ defines a trace preserving automorphism of $\cZ(A)$. When $v \in M$ is a partial isometry with $p = v^* v \in A$, $q = v v^* \in A$ and $v A v^* = q A q$, the restriction of $\Ad v$ to the center of $A$ defines a partial automorphism of $\cZ(A)$ that need not be trace preserving. Writing $\cZ(A) = \rL^\infty(X, \mu)$, it even happens quite naturally that the orbit equivalence relation induced by the orbits of all these $\Ad v$ is a type III equivalence relation on $(X,\mu)$ (see \cite{MV13} for an example where this phenomenon occurs).

Note that the stable normalizer of a von Neumann subalgebra $Q \subset M$ is contained in the quasi-normalizer, defined as the von Neumann algebra generated by
$$\QN_M(Q) = \Bigl\{ x \in M \Bigm| \exists x_i,y_j \in M , x Q \subset \sum_{i=1}^n Q x_i \;\;\text{and}\;\; Q x \subset \sum_{j=1}^m y_j Q \Bigr\} \; .$$
If $Q \subset M$ is with expectation, then $\QN_M(Q)\dpr \subset M$ is also with expectation. The proof is entirely analogous to the one showing that $\sN_M(Q)\dpr \subset M$ is with expectation. In general, the inclusion $\sN_M(Q)\dpr \subset \QN_M(Q)\dpr$ is strict. Nevertheless, when $Q$ is abelian, we have the following result.

\begin{prop}\label{prop.quasi-normalizer}
Let $M$ be a stably strongly solid von Neumann algebra and $A \subset M$ a diffuse abelian von Neumann subalgebra with expectation. Then $\QN_M(A)\dpr$ is amenable.
\end{prop}
\begin{proof}
Since $A$ is abelian, $\QN_M(A)\dpr$ is generated by partial isometries $v \in M$ with right support $p = v^* v$ and left support $q = v v^*$ belonging to $A'\cap M$ and with $v A v^* = A q$. Writing $Q = A \vee (A' \cap M)$, it follows that $\QN_M(A)\dpr \subset \sN_M(Q)\dpr$. Since $M$ is in particular solid, $Q$ is diffuse and amenable. By stable strong solidity and since $Q \subset M$ is with expectation, we conclude that $\sN_M(Q)\dpr$ is amenable and thus also $\QN_M(A)\dpr$ is amenable since $\QN_M(A)\dpr \subset \sN_M(Q)\dpr$ is with expectation.
\end{proof}

\subsection{Properties of the stable normalizer}

The set $\sN_M(Q)$ behaves well under amplifications/reductions and is itself a stable version of the normalizer $\cN_M(Q)$. In \cite[Lemma 2.1]{JP81} and \cite[Lemma 3.5]{Po03}, a detailed analysis of normalizing unitaries versus amplifications/reductions was made and several key techniques were introduced. We only need the following easy and well known lemma (see e.g.\ \cite[Lemma 3.2]{FSW10} for a proof of the last statement, also based on \cite[Lemma 2.1]{JP81}).

\begin{lem}\label{lem.prop-PN}
Let $M$ be a von Neumann algebra and $Q \subset M$ a von Neumann subalgebra.
\begin{itemize}
\item [$(1)$] For every projection $q \in Q$, we have $q \sN_M(Q) q = \sN_{qMq}(qQq)$.

\item [$(2)$] For every Hilbert space $K$, we have $\sN_{M \ovt \rB(K)}(Q \ovt \rB(K))\dpr = \sN_M(Q)\dpr \ovt \rB(K)$.

\item [$(3)$] If $Q$ is $\sigma$-finite and $K = \ell^2(\N)$, we have
$$\sN_M(Q)\dpr \ovt \rB(K) = \cN_{M \ovt \rB(K)}(Q \ovt \rB(K))\dpr \; .$$
\end{itemize}
\end{lem}

We also need the following technical lemma providing an explicit dilation of a partial isometry in $\sN_M(Q)$ to a normalizing unitary of an infinite amplification of $Q$. We need this explicit version to get as a conclusion that for every amenable $Q$ and every fixed $x \in \sN_M(Q)$, the von Neumann algebra $(Q \cup \{x,x^*\})\dpr$ remains amenable. Again the method of proof is basically given by \cite[Lemma 2.1]{JP81}.

\begin{lem}\label{lem.amplify}
Let $M$ be a von Neumann algebra and $Q \subset M$ a $\sigma$-finite von Neumann subalgebra. Define $\Qtil = Q \ovt \rB(\ell^2(\N)) \ovt \ell^\infty(\N)$ and view $\Qtil$ as a von Neumann subalgebra of $\Mtil = M \ovt \rB(\ell^2(\N \times \N))$. Denote by $e \in \rB(\ell^2(\N \times \N))$ the minimal projection given by the unit vector $\delta_0 \ot \delta_0$.

For every $x \in \sN_M(Q)$, there exist $u \in \cN_{\Mtil}(\Qtil)$ and $a,b \in Q$ such that
$$x \ot e = u(a \ot e) = (b \ot e)u \; .$$
If $Q$ is amenable and $x \in \sN_M(Q)$, then $(Q \cup \{x,x^*\})\dpr$ remains amenable.
\end{lem}

\begin{proof}
Denote by $e_{00} \in \rB(\ell^2(\N))$ the minimal projection given by the unit vector $\delta_0$. By taking the polar decomposition of $x$, we may assume that $x = v$ is a partial isometry with $p = v^*v$ and $q = vv^*$ belonging to $Q$ and $v Q v^* = q Q q$. We prove that there exists a $u \in \cN_{\Mtil}(\Qtil)$ such that $v \ot e = u(p \ot e) = (q \ot e)u$. Denote by $z_p$ the central support of $p$ in $Q$ and similarly define $z_q$. Define $Q_1 = Q \ovt \rB(\ell^2(\N))$ and view $Q_1$ as a von Neumann subalgebra of $M_1 = M \ovt \rB(\ell^2(\N))$. In $Q_1$, the projection $z_p \ot 1$ is equivalent with infinitely many copies of $p \ot e_{00}$. So we can take a sequence of partial isometries $w_n \in Q_1$ such that $w_n^* w_n = p \ot e_{00}$ for all $n$ and $\sum_n w_n w_n^* = z_p \ot 1$. We make this choice such that $w_0 = p \ot e_{00}$. Similarly take a sequence of partial isometries $v_n \in Q_1$ such that $v_n^* v_n = q \ot e_{00}$ and $\sum_n v_n v_n^* = z_q \ot 1$. Also here, we make this choice such that $v_0 = q \ot e_{00}$. Define $V = \sum_n v_n (v \ot e_{00}) w_n^*$. Note that $V^* V = z_p \ot 1$, $VV^* = z_q \ot 1$ and $V Q_1 V^* = Q_1 (z_q \ot 1)$. By construction, $V (p \ot e_{00}) = v \ot e_{00} = (q \ot e_{00}) V$.

Identify $\cZ(Q) = \rL^\infty(X, \mu)$ for some probability space $(X,\mu)$. Take $\cU,\cV \subset X$ such that $z_p = 1_\cU$ and $z_q = 1_\cV$. The restriction of $\Ad V$ to $\cZ(Q_1) = \cZ(Q) \ot 1$ induces a nonsingular transformation $\vphi : \cU \recht \cV$. Denote by $\cR$ the countable nonsingular equivalence relation on $(X,\mu)$ generated by the graph of $\vphi$. Using powers of $V$ and $V^*$, it follows that every partial transformation $\psi$ in the full pseudogroup of $\cR$ is given by the restriction of $\Ad W$ to $\cZ(Q_1)$ for some partial isometry $W \in M_1$ such that $z = W^*W$ and $z'=WW^*$ belong to $\cZ(Q_1)$ and $W Q_1 W^* = Q_1 z'$.

Define the equivalence relation $\cRtil$ on $\Xtil = X \times \N$ given by $(x,i) \cRtil (y,j)$ iff $x \cR y$. We say that two Borel subsets of $\Xtil$ are equivalent if they are, up to measure zero, the range (resp.\ domain) of an element of the full pseudogroup $[[\cRtil]]$. By construction, the sets $\cU \times \{0\}$ and $\cV \times \{0\}$ are equivalent. Being properly infinite with $\cRtil$-saturations equal to $\Xtil$, also the sets $\Xtil \setminus (\cU \times \{0\})$ and $\Xtil \setminus (\cV \times \{0\})$ are equivalent. This means that we can find a partial isometry $W \in \Mtil$ such that $W^* W = 1 - z_p \ot 1 \ot e_{00}$, $W W^* = 1 - z_q \ot 1 \ot e_{00}$ and $W \Qtil W^* = \Qtil (1 - z_q \ot 1 \ot e_{00})$.

The unitary $u = V \otimes e_{00} + W$ belongs to $\cN_{\Mtil}(\Qtil)$ and satisfies $v \ot e = u(p \ot e) = (q \ot e)u$.

Finally assume that $Q$ is amenable and $x \in \sN_M(Q)$. First, replace $M$ by $(Q \cup \{x,x^*\})\dpr$. Then, take $u$ as above. Since $\Qtil$ is amenable and $u$ normalizes $\Qtil$, also $\Ptil := (\Qtil \cup \{u,u^*\})\dpr$ is amenable. Since the von Neumann algebra $(1 \ot e)\Ptil (1 \ot e)$ contains $Q$ and $x$, it must be equal to $M \ot e$ and we conclude that $M$ is amenable.
\end{proof}

\subsection{A general weak compactness argument}

In \cite[Theorem 3.5]{OP07}, Ozawa and Popa proved the following seminal result: if a tracial von Neumann algebra $(M,\tau)$ has the complete metric approximation property (CMAP) and $A \subset M$ is any amenable von Neumann subalgebra, then the action of the normalizer $\cN_M(A)$ on $A$ given by conjugacy is \emph{weakly compact} in the sense of \cite[Definition 3.1]{OP07}, meaning that there exists a net of unit vectors $\xi_n \in \rL^2(A \ovt A\op)$ such that
\begin{itemize}
\item $\lim_n \| \xi_n - (a \ot \abar) \xi_n\|_2 = 0$ for all $a \in \cU(A)$,

\item $\lim_n \| \xi_n - \Ad(u \ot \ubar) \xi_n \|_2 = 0$ for all $u \in \cN_M(A)$,

\item $\lim_n \langle (a \ot 1) \xi_n , \xi_n \rangle = \tau(a)$ for all $a \in A$.
\end{itemize}
Here, we denote by $A\op$ the opposite von Neumann algebra of $A$. We also denote $\abar = (a^*)\op$ for every $a \in A$.

We adapt the proof of \cite[Theorem 3.5]{OP07} to also cover conjugation by elements $x \in \sN_M(A)$.

Fix a tracial von Neumann algebra $(M,\tau)$ and a von Neumann subalgebra $A \subset M$. Denote by $\rE_\cZ$ the unique trace preserving conditional expectation of $A$ onto $\cZ(A)$.
For every $x \in \sN_M(A)$, we define $z^r_x$ as the support projection of $\rE_\cZ(x^* x)$ and we define $z^l_x$ as the support of $\rE_\cZ(xx^*)$. Write $x = v |x|$ for the polar decomposition of $x \in \sN_M(A)$ and let $p = v^*v$ and $q = vv^*$. Then we have $|x| \in A$ and $vAv^* = qAq$ so that $v \in \sN_M(A)$. Note that we have $z_v^l = z_x^l$ and $z_v^r = z_x^r$. We denote by $\al_v : \cZ(A) z^r_v \recht \cZ(A) z^l_v$ the unique $*$-isomorphism determined by $v a = \al_v(a) v$ for all $a \in \cZ(A) z^r_v$. Multiplying by $|x| \in A$ on the right hand side, we obtain $x a = v |x| a =  v a |x| = \al_v(a) v |x| = \al_v(a) x$ for all $a \in \cZ(A) z^r_v$. Letting $\alpha_x = \alpha_v$, we obtain $x a = \al_x(a) x$ for all $a \in \cZ(A) z^r_x$. 

The main difficulty comes from the fact that $\al_x$ need not be trace preserving. We denote by $\Delta_x$ the Radon--Nikodym derivative between $\tau$ and $\tau \circ \al_x$, i.e.\ $\Delta_x$ is the unique positive self-adjoint nonsingular operator affiliated with $\cZ(A) z^l_x$ satisfying $\tau(\Delta_x \al_x(a)) = \tau(a)$ for all $a \in \cZ(A) z^r_x$. Note that $\Delta_{xy} = \Delta_x \al_x(\Delta_y)$ for all $x, y \in \sN_M(A)$. Also note that $\Delta_x = \rE_\cZ(xx^*) \, \al_x(\rE_\cZ(x^*x)^{-1})$.

We need the following notation.
$$\sN_M^0(A) = \{ x \in \sN_M(A) \mid \; \exists \delta > 0 \;\;\text{such that}\;\; \rE_\cZ(x^* x) \geq \delta z^r_x \;\;\text{and}\;\; \rE_\cZ(xx^*) \geq \delta z^l_x \; \} \; .$$
Note that for every $x \in \sN_M(A)$, we can choose a sequence of projections $z_n \in \cZ(A)$ such that $z_n \recht 1$ strongly and $x z_n \in \sN_M^0(A)$ for every $n$. In particular, $\sN_M^0(A)$ generates the same von Neumann algebra as $\sN_M(A)$. For every $x \in \sN_M^0(A)$, the Radon--Nikodym derivative $\Delta_x$ is a bounded invertible operator in $\cZ(A) z^l_x$.

\begin{prop}[Weak compactness]\label{prop.WC}
Let $(M,\tau)$ be a tracial von Neumann algebra with the CMAP and take an amenable von Neumann subalgebra $A \subset M$. Then there exists a net of positive vectors $\xi_n \in \rL^2(A \ovt A\op)$ such that
\begin{itemize}
\item [$(1)$] $\lim_n \Vert (a \ot 1)\xi_n - (1 \ot a\op)\xi_n \Vert_2 = 0$, for all $a \in A$~;

\item [$(2)$] $\lim_n \Vert (x \ot 1)\xi_n(x^*\Delta_x^{1/2} \ot 1) - (1 \ot x\op)\xi_n(1 \ot \xbar)\Vert_2 = 0$, for all $x \in \sN_M^0(A)$~;

\item [$(3)$] $\lim_n \langle (x \ot 1)\xi_n,\xi_n \rangle = \tau(x)$, for all $x \in M$.
\end{itemize}
Also, for every partial isometry $v \in \sN_M^0(A)$, there exists an element $T(v)$ in the unit ball of $M \ovt M\op$ and a sequence of elements $T(v,k)$ in the unit ball of $M \otalg M\op$ such that
\begin{equation}\label{eq.Tv}
\begin{split}
& \lim_n \|(v \ot 1) \xi_n - (1 \ot v\op) \xi_n T(v) \|_2 = 0 \;\; , \\
& \lim_n \|(v^* \ot 1) \xi_n - (1 \ot \vbar) \xi_n T(v)^* \|_2 = 0 \;\; ,\\
& \lim_k \Bigl(\limsup_n \|(v \ot 1) \xi_n - (1 \ot v\op) \xi_n T(v,k) \|_2 \Bigr) = 0 \;\; ,\\
& \lim_k \Bigl(\limsup_n \|(v^* \ot 1) \xi_n - (1 \ot \vbar) \xi_n T(v,k)^* \|_2 \Bigr) = 0 \;\; .
\end{split}
\end{equation}
\end{prop}
\begin{proof}
Since $M$ has CMAP, we can take a net of finite rank, normal, completely bounded maps $\vphi_n : M \recht M$ such that $\lim_n \|\vphi_n\|_{\cb} = 1$ and $\lim_n \|\vphi_n(x) - x\|_2 = 0$ for all $x \in M$. Exactly as in the proof of \cite[Theorem 3.5]{OP07}, we then define, for every amenable von Neumann subalgebra $Q \subset M$, the normal functionals $\mu_n^Q \in (Q \ovt Q\op)_*$ given by $\mu_n^Q(a \ot b\op) = \tau(\vphi_n(a) b)$ for all $a,b \in Q$ and satisfying $\lim_n \|\mu_n^Q\| = 1$. We define the normal states $\om_n^Q \in (Q \ovt Q\op)_*$ given by $\om_n^Q = \|\mu_n^Q\|^{-1} \, |\mu_n^Q|$. Since $\lim_n \mu_n^Q(a \ot \abar) = 1$ for all $a \in \cU(Q)$, we get, still in the same way as in the proof of \cite[Theorem 3.5]{OP07}, that $\lim_n \|\mu_n^Q - \om_n^Q\| = 0$, $\lim_n \| (a \ot \abar) \cdot \om_n^Q - \om_n^Q\| = 0$ and $\lim_n \|\om_n^Q \cdot (a \ot \abar) - \om_n^Q\|=0$ for all $a \in \cU(Q)$. This implies that
\begin{equation}\label{eq.prop-om-n-Q}
\lim_n \|(a \ot 1) \cdot \om_n^Q - (1 \ot a\op) \cdot \om_n^Q \| = 0 \quad\text{and}\quad \lim_n \|\om_n^Q \cdot (a \ot 1) - \om_n^Q \cdot (1 \ot a\op) \| = 0
\end{equation}
for all $a \in Q$.

We view $\om_n^Q$ as an element of $\rL^1(Q \ovt Q\op)^+$. We define $\xi_n = (\om_n^A)^{1/2}$ and prove that the net $\xi_n \in \rL^2(A \ovt A\op)^+$ satisfies the conclusions of the proposition.

Properties (1) and (3) hold immediately. Since we already have (1), it suffices to prove property (2) when $x = v$ is a partial isometry in $\sN_M^0(A)$. Fix such a $v$ and write $q = vv^*$, $p = v^* v$. Define $D_r = (\rE_\cZ(p))^{1/2}$ and $D_l = (\rE_\cZ(q))^{1/2}$. Denote by $z_r,z_l$ the support projections of $D_r, D_l$. Then, $D_r$ (resp.\ $D_l$) are invertible operators in $\cZ(A)z_r$ (resp.\ $\cZ(A)z_l$) and $\Delta_v^{1/2} = D_l \, \al_v(D_r^{-1})$. Put $Q = (A \cup \{v,v^*\})\dpr$. By Lemma \ref{lem.amplify}, $Q$ is amenable. It then follows from \eqref{eq.prop-om-n-Q} that
\begin{equation}\label{eq.with-Q}
\lim_n \|(v \ot 1) \cdot \om_n^Q \cdot (v^* \ot 1) - (1 \ot v\op) \cdot \om_n^Q \cdot (1 \ot \vbar) \|_1 = 0 \; .
\end{equation}
The restriction of $\mu_n^Q$ to $A \ovt A\op$ equals $\mu_n^A$. Therefore, $\lim_n \|\rE_{A \ovt A\op}(\om_n^Q) - \om_n^A \|_1 = 0$. Because $v \in \sN_M(A)$, we have that $\rE_A(v y v^*) = v \rE_A(y) v^*$ for all $y \in M$. Applying $\rE_{A \ovt A\op}$ to \eqref{eq.with-Q}, we conclude that
\begin{equation}\label{eq.estim-L1}
\lim_n \|(v \ot 1) \cdot \om_n^A \cdot (v^* \ot 1) - (1 \ot v\op) \cdot \om_n^A \cdot (1 \ot \vbar) \|_1 = 0 \; .
\end{equation}

By Lemma \ref{lem.equiv} below, we can take sequences of elements $a_i,b_j \in A$ such that
\begin{equation}\label{eq.aibj}
\sum_{i=0}^\infty a_i a_i^* = z_l \quad , \quad \sum_{i=0}^\infty a_i^* a_i = D_l^{-2} \, q \quad , \quad \sum_{j=0}^\infty b_j b_j^* = z_r \quad , \quad \sum_{j=0}^\infty b_j^* b_j = D_r^{-2} \, p \; .
\end{equation}
We make this choice such that $a_0 = q$ and $b_0 = p$.

Consider the von Neumann algebra $\cN = \rB(\ell^2(\N^2) \oplus \C) \ovt M \ovt M\op$ with its canonical semifinite trace $\Tr \ot \tau$ and associated $1$-norm $\|\,\cdot\,\|_1$ and $2$-norm $\|\,\cdot\,\|_2$. View $\rB(\C, \ell^2(\N^2)) \subset \rB(\ell^2(\N^2) \oplus \C)$ and denote by $e_{ij} \in \rB(\C, \ell^2(\N^2))$ the operator given by $e_{ij}(\mu) = \mu \delta_{ij}$, where $(\delta_{ij})_{i,j \in \N}$ is the canonical orthonormal basis of $\ell^2(\N^2)$. We will identify $\rB(\C, \C ) \ovt M \ovt M\op = \C e_\C \ovt M \ovt M\op$ with $M \ovt M\op$. Observe that this identification preserves the $1$-norm $\|\,\cdot\,\|_1$ and the $2$-norm $\|\,\cdot\,\|_2$.

Define $V \in \rB(\C, \ell^2(\N^2)) \ovt M \ovt 1$ given by
$$V = \sum_{i,j} e_{ij} \ot D_l a_i v b_j^* \ot 1 \; .$$
Note that $V$ is a well defined bounded operator satisfying $V^* V = z_r \ot 1$. We similarly define $W \in \rB(\C, \ell^2(\N^2)) \ovt 1 \ovt M\op$ given by
$$W = \sum_{i,j} e_{ij} \ot 1 \ot (a_i v b_j^* D_r)\op$$
and note that $W^* W = 1 \ot z_l\op$.

We claim that \eqref{eq.estim-L1} together with properties (1) and (3) implies that
\begin{equation}\label{eq.claim}
\lim_n \| V (D_r^2 \ot 1) \, \om_n^A \, V^* - W (1 \ot (D_l^2)\op) \, \om_n^A \, W^* \|_1 = 0 \; .
\end{equation}
For every finite subset $\cF \subset \N^2$, we define $V_\cF$ and $W_\cF$ in the same way as $V$ and $W$ by only summing over $(i,j) \in \cF$. Note that $\|V_\cF\| \leq 1$ for all $\cF \subset \N^2$ and note that $\|V - V_\cF\|_2$ can be made arbitrarily small. For all $U_1, U_2 \in \rB(\C, \ell^2(\N^2)) \ovt M \ovt 1$, we find using the Cauchy--Schwarz inequality and property (3) that
$$\limsup_n \| U_1 \, \om_n^A \, U_2^* \|_1 \leq \limsup_n \| U_1 \, \xi_n \|_2 \, \|U_2 \, \xi_n\|_2 = \|U_1\|_2 \|U_2\|_2 \; .$$
The same inequality holds for all $U_1, U_2 \in \rB(\C, \ell^2(\N^2)) \ovt 1 \ovt M\op$.
From \eqref{eq.estim-L1} and property (1), we immediately get that for every finite subset $\cF \subset \N^2$,
$$\lim_n \| V_\cF (D_r^2 \ot 1) \, \om_n^A \, V_\cF^* - W_\cF (1 \ot (D_l^2)\op) \, \om_n^A \, W_\cF^* \|_1 = 0 \; .$$
By the preceding discussion, we conclude that also \eqref{eq.claim} holds.

Because $D_r$ belongs to the center of $A$ and $V^* V = z_r \ot 1$, the element $V (D_r \ot 1) \, \xi_n \, V^*$ is the positive square root of $V (D_r^2 \ot 1) \, \om_n^A \, V^*$. Similarly, $W (1 \ot D_l\op) \, \xi_n \, W^*$ is the positive square root of $W (1 \ot (D_l^2)\op) \, \om_n^A \, W^*$. The Powers--St{\o}rmer inequality in $\mathcal N$ then implies that
\begin{equation}\label{eq.estim-L2}
\lim_n \| V (D_r \ot 1) \, \xi_n \, V^* - W (1 \ot D_l\op) \, \xi_n \, W^* \|_2 = 0 \; .
\end{equation}
Multiplying on the left with $e_{00}^* \ot 1 \ot 1$ and on the right with $e_{00} \ot 1 \ot 1$, we find that 
\begin{equation*}
\lim_n \| (D_l v D_r \ot 1) \, \xi_n \, (v^*D_l  \ot 1)  -  (1 \ot (D_l v D_r)\op) \, \xi_n \, (1 \otimes \overline{v D_r}) \|_2 = 0 \; .
\end{equation*}
Since $v D_r = \alpha_v(D_r) v$, using $(1)$ and multiplying on the left with $D_l^{-1} \otimes (D_r^{-1})\op$ and on the right with $\alpha_v(D_r^{-1}) \otimes 1$, we find that 
\begin{equation*}
\lim_n \| ( v \ot 1) \, \xi_n \, (v^*D_l \alpha_v(D_r^{-1}) \ot 1)  -  (1 \ot v\op) \, \xi_n \, (1 \otimes \vbar) \|_2 = 0 \; .
\end{equation*}
Since $\Delta_v^{1/2} = D_l \alpha_v(D_r^{-1})$, we finally obtain $(2)$. Denote $T(v) = W^* V$. Then $T(v)$ belongs to the unit ball of $M \ovt M\op$. Multiplying \eqref{eq.estim-L2} on the left with $e_{00}^* \ot 1 \ot 1$ and on the right with $V$, we find that 
\begin{equation*}
\lim_n \| (D_l v D_r \ot 1) \, \xi_n -  (1 \ot (D_l v D_r)\op) \, \xi_n \, T(v) \|_2 = 0 \; .
\end{equation*}
Using $(1)$ and multiplying on the left with $D_l^{-1} \otimes (D_r^{-1})\op$, we find that the first estimate in \eqref{eq.Tv} holds.

Define $V_k$ by the same formula as $V$, but only summing over $i,j=1,\ldots,k$. Define $T(v,k) = W^* V_k$ and note that $T(v,k)$ belongs to the unit ball of $M \otalg M\op$. Write $V^* V_k = d_k \ot 1$, where $d_k \in A$ and $\lim_k \|z_r - d_k\|_2 = 0$. Multiplying \eqref{eq.estim-L2} on the left with $e_{00}^* \ot 1 \ot 1$ and on the right with $V_k$ and reasoning as before, we find that
$$\limsup_n \| (v \ot 1) \, \xi_n \, (d_k \ot 1) - (1 \ot v\op) \, \xi_n \, T(v,k)\|_2 = 0 \; .$$
Since $\lim_n \| \xi_n (d_k \ot 1) - \xi_n (z_r \ot 1)\|_2 = \|d_k - z_r\|_2$, we conclude that also the third estimate in \eqref{eq.Tv} holds.

Replacing in the above reasoning $v$ by $v^*$, in \eqref{eq.aibj}, we interchange to roles of $a_i$ and $b_j$. We then get that
\begin{align*}
T(v^*) &= \sum_{i,j} D_r b_j v^* a_i^* \ot (D_l a_i v b_j^*)\op = (D_r \ot (D_r^{-1})\op) \, T(v)^* \, (D_l^{-1} \ot D_l\op) \quad\text{and}\\
T(v^*,k) &= (D_r \ot (D_r^{-1})\op) \, T(v,k)^* \, (D_l^{-1} \ot D_l\op) \; .
\end{align*}
Since $\lim_n \| (1 \otimes \vbar)\xi_n (D_r \ot (D_r^{-1})\op) - (1 \otimes \vbar)\xi_n\|_2 = 0$ and $\lim_n \|(v^* \otimes 1)\xi_n (D_l^{-1} \ot D_l\op) - \linebreak (v^* \otimes 1)\xi_n \|_2 = 0$,
it follows that also the second and fourth estimate in \eqref{eq.Tv} hold.
\end{proof}

\subsection{Consequences of weak compactness}

The approximate invariance given by weak compactness as in \eqref{eq.Tv} combines very well with deformation/rigidity theory. In particular, we can apply to $\xi_n$ any $s$-malleable deformation of $M$, in the sense of \cite{Po03}, i.e.\ a trace preserving inclusion $(M,\tau) \subset (\Mtil,\tau)$ together with a strongly continuous one-parameter group of trace preserving automorphisms $(\al_t)_{t \in \R}$ of $\Mtil$ and a trace preserving period $2$ automorphism $\be \in \Aut(\Mtil,\tau)$ satisfying $\be \circ \al_t = \al_{-t} \circ \be$ and $\be|_M = \id$.

Following \cite[Definition 2.3]{PV11}, for any tracial von Neumann algebras $P \subset (M, \tau)$ and $(Q, \tau)$, we say that an $M$-$Q$-bimodule $\vphantom{}_M \mathcal K_Q$ is {\em left $P$-amenable} if there exists a $P$-central state $\Omega$ on $\rB(\mathcal K) \cap (Q\op)'$ whose restriction to $M$ equals $\tau$. The methods of \cite[Section 4]{OP07} can be applied and give the following result.

\begin{prop}\label{prop.malleable}
Let $(M,\tau)$ be a tracial von Neumann algebra with the CMAP and $A \subset M$ a von Neumann subalgebra. Assume that $\xi_n \in \rL^2(A \ovt A\op)$ is a net of positive vectors satisfying the conclusion of Proposition \ref{prop.WC}. Also assume that $(\al_t)_{t \in \R}$ is an $s$-malleable deformation as above. Then at least one of the following statements holds.
\begin{itemize}
\item [$(1)$] We have that $\lim_{t \recht 0} \bigl(\sup_{a \in \cU(A)} \|\al_t(a) - a\|_2 \bigr) = 0$.

\item [$(2)$] Writing $P = \sN_M(A)\dpr$, there exists a nonzero projection $p \in \cZ(P)$ such that the $pMp$-$M$-bimodule $p \rL^2(\Mtil \ominus M)$ is left $Pp$-amenable.
\end{itemize}
\end{prop}

\begin{proof}
Denote by $e_M : \rL^2(\Mtil) \recht \rL^2(M)$ the orthogonal projection and write $e_M^\perp = 1 - e_M$. One of the following properties holds.
\begin{enumerate}
\item For every $\eps > 0$, there exists a $t_0 > 0$ such that for every $t \in \R$ with $|t| \leq t_0$, we have $\limsup_n \|(e_M^\perp \al_t \ot \id)(\xi_n)\|_2 \leq \eps$.
\item There exists an $\eps > 0$ and a sequence $t_k \in \R$ such that $\lim_k t_k = 0$ and such that for every $k$, we have $\limsup_n \| (e_M^\perp \al_{t_k} \ot \id)(\xi_n)\|_2 > \eps$.
\end{enumerate}
We prove that the first (resp.\ second) of these properties implies the first (resp.\ second) conclusion in the proposition.

Assume that (1) holds. By \cite[Lemma 2.1]{Po06}, the following transversality condition holds for every $t \in \R$ and $\xi \in \rL^2(M \ovt M\op)$.
$$\|(\al_{2t} \ot \id)(\xi) - \xi \|_2 \leq 2 \| (e_M^\perp \al_t \ot \id)(\xi)\|_2 \; .$$
Choose $\eps > 0$. We can then take a $t_0 > 0$ such that for every $t \in \R$ with $|t| \leq t_0$, we have $\limsup_n \|(\al_t \ot \id)(\xi_n) - \xi_n\|_2 \leq \eps$. Fix $t \in \R$ with $|t| \leq t_0$. We prove that $\|\al_t(a)-a\|_2 \leq 2 \sqrt{\eps}$ for every $a \in \cU(A)$, so that the first conclusion of the proposition indeed holds. Fix $a \in \cU(A)$. Because
$$\lim_n \langle (\al_t(a) \ot \abar) \, (\al_t \ot \id)(\xi_n) , (\al_t \ot \id)(\xi_n) \rangle = 1$$
and because $\limsup_n \|(\al_t \ot \id)(\xi_n) - \xi_n\|_2 \leq \eps$, we get that
$$\limsup_n | 1 - \langle (\al_t(a) \ot \abar) \xi_n , \xi_n \rangle | \leq 2\eps \; .$$
But the left hand side equals
$$\limsup_n | 1 - \langle (\al_t(a) a^* \ot 1) \xi_n , \xi_n \rangle | = | 1 - \tau(\al_t(a) a^*) | \; .$$
So, we have proved that $|1 - \tau(\al_t(a) a^*)| \leq 2 \eps$. Then also,
$$\| \al_t(a) - a \|_2^2 = 2 \real( 1- \tau(\al_t(a) a^*)) \leq 4 \eps \; .$$

Next assume that (2) holds with $\varepsilon > 0$ and $(t_k)_k$. We start by proving the following claim~: for every $x \in M$ and every $\delta > 0$, we have that for small enough $t \in \R$
$$\limsup_n \| (x \ot 1) (\al_t \ot \id)(\xi_n) - (\al_t \ot \id)((x \ot 1) \xi_n) \|_2 < \delta \; .$$
Indeed, it suffices to observe that the left hand side equals
$$\limsup_n \| ((\al_{-t}(x) - x) \ot 1) \xi_n \|_2 = \|\al_{-t}(x) - x\|_2$$
and that $\|\al_{-t}(x) - x\|_2 \recht 0$ as $t \recht 0$. Then also for every $T \in M \otalg M\op$ and every $\delta > 0$, we have that for small enough $t \in \R$
$$\limsup_n \| (\al_t \ot \id)(\xi_n) T - (\al_t \ot \id)(\xi_n T)\|_2 < \delta \; .$$

We construct a subnet $\zeta_i$ of the net $\zeta_{j,n} = (e_M^\perp \al_{t_j} \ot \id)(\xi_n)$ such that $\eps \leq \|\zeta_i\|_2 \leq 1$ for every $i$ and such that for every partial isometry $v \in \sN_M^0(A)$ and every $x \in M$, we have
\begin{equation}\label{eq.net-zeta-i}
\begin{split}
& \lim_i \|(v \ot 1) \zeta_i - (1 \ot v\op) \zeta_i S(v,i) \|_2 = 0 \;\; , \;\; \lim_i \|(v^* \ot 1) \zeta_i - (1 \ot \vbar) \zeta_i S(v,i)^* \|_2 = 0 \;\; , \\
& \limsup_i \| (x \ot 1) \zeta_i \|_2 \leq \|x\|_2 \;\; ,
\end{split}
\end{equation}
where the $S(v,i)$ are elements in the unit ball of $M \otalg M\op$. The index set of the net $\zeta_i$ is given by $i = (\cF,\cG,\delta)$ where $\cF$ is a finite set of partial isometries
in $\sN_M^0(A)$, $\cG \subset M$ is a finite subset and $\delta > 0$. Given $i = (\cF,\cG,\delta)$ and using the notation of \eqref{eq.Tv}, we take $k$ large enough such that for every $v \in \cF$, we have that
$$\limsup_n \|(v \ot 1) \xi_n - (1 \ot v\op) \xi_n T(v,k) \|_2 < \delta \quad\text{and}\quad \limsup_n \|(v^* \ot 1) \xi_n - (1 \ot \vbar) \xi_n T(v,k)^* \|_2 < \delta \;\;.$$
Using the claim in the previous paragraph, we then take $j$ large enough such that
$$\limsup_n \| (x \ot 1) (\al_{t_j} \ot \id)(\xi_n) - (\al_{t_j} \ot \id)((x \ot 1) \xi_n) \|_2 < \delta$$
for all $x \in \cF \cup \cG$ and such that
$$\limsup_n \| (\al_{t_j} \ot \id)(\xi_n) T(v,k) - (\al_{t_j} \ot \id)(\xi_n T(v,k))\|_2 < \delta$$
for all $v \in \cF$. We finally take $n$ large enough such that the vector $\zeta_i = (e_M^\perp \al_{t_j} \ot \id)(\xi_n)$ together with $S(v,i) := T(v,k)$ satisfies $\eps \leq \|\zeta_i\|_2 \leq 1$ and
\begin{align*}
& \|(v \ot 1) \zeta_i - (1 \ot v\op) \zeta_i S(v,i) \|_2 < 3\delta \quad , \quad \|(v^* \ot 1) \zeta_i - (1 \ot \vbar) \zeta_i S(v,i)^* \|_2 < 3\delta \quad ,\\
& \| (x \ot 1) \zeta_i \|_2 < \|x\|_2 + 2 \delta
\end{align*}
for all $v \in \cF$ and all $x \in \cG$. So we have found the net $\zeta_i$ satisfying \eqref{eq.net-zeta-i}.

Denote $P = \sN_M(A)\dpr$ and define $\cS$ as the commutant of the right $M$-action on $\rL^2(\Mtil \ominus M)$. Taking a subnet of the net $\zeta_i$, we may assume that the net of positive functionals $\cS \recht \C : T \mapsto \langle (T \ot 1) \zeta_i, \zeta_i \rangle$ converges weakly$^*$ to a positive functional $\Om$ on $\cS$. By \eqref{eq.net-zeta-i}, we get for all $T \in \cS$ and all partial isometries $v \in \sN_M^0(A)$ that
\begin{align*}
\Om(T v) & = \lim_i \langle (T v \ot 1) \zeta_i , \zeta_i \rangle = \lim_i \langle (T \ot v\op) \zeta_i S(v,i) , \zeta_i \rangle \\
& = \lim_i \langle (T \ot 1) \zeta_i , (1 \ot \vbar) \zeta_i S(v,i)^* \rangle = \lim_i \langle (T \ot 1) \zeta_i, (v^* \ot 1) \zeta_i \rangle \\
& = \lim_i \langle (v T \ot 1) \zeta_i, \zeta_i \rangle = \Om(vT) \;\;.
\end{align*}
Since $\|\zeta_i\|_2 \geq \eps$ for all $i$, also $\Om(1) \geq \eps$, so that $\Om$ is nonzero. By \eqref{eq.net-zeta-i}, we get that $\Om(x^* x) \leq \tau(x^*x)$ for all $x \in M$. The Cauchy--Schwarz inequality then implies that
$$|\Om(T v) - \Om(T w)|^2 \leq \Om(TT^*) \, \|v-w\|_2^2 \quad\text{and}\quad |\Om(v T - w T)|^2 \leq \|v-w\|_2^2 \, \Om(T^* T)$$
for all $T \in \cS$ and $v,w \in M$. Since $\Om(T v) = \Om(v T)$ when $v$ is a partial isometry in $\sN_M^0(A)$, taking linear combinations and $\|\,\cdot\,\|_2$-limits, it follows that $\Om$ is a nonzero $P$-central functional on $\cS$. Since $\Om|_M \leq \tau$, the restriction of $\Om$ to $M$ is normal. Denote by $p \in M$ the support projection of $\Om|_M$. Then, $p \in P' \cap M$. Since $A \subset P$ and $A'\cap M \subset P$, it follows that $p \in \cZ(P)$. As in \cite[Theorem 2.1]{OP07}, we conclude that the $pMp$-$M$-bimodule $p \rL^2(\Mtil \ominus M)$ is left $Pp$-amenable.
\end{proof}

Propositions \ref{prop.WC} and \ref{prop.malleable} will be key to prove that free Araki--Woods factors are strongly solid. We also have the following consequence.

\begin{thm}\label{thm.stable-strong-solidity}
Let $\Gamma$ be a countable group with the CMAP that admits a proper $1$-cocycle into an orthogonal representation that is weakly contained in the regular representation. Then, $M = \rL(\Gamma)$ is stably strongly solid: for every diffuse amenable von Neumann subalgebra $A \subset M$, we have that $\sN_M(A)\dpr$ remains amenable.
\end{thm}
\begin{proof}
This follows immediately from Propositions \ref{prop.WC} and \ref{prop.malleable} by using the $s$-malleable deformation associated with a $1$-cocycle in \cite{Si10}.
\end{proof}

We needed the following well known lemma.

\begin{lem}\label{lem.equiv}
Let $(A,\tau)$ be a tracial von Neumann algebra. Denote by $\rE_\cZ : A \recht \cZ(A)$ the unique trace preserving conditional expectation. If $x,y \in A^+$ satisfy $\rE_\cZ(x) = \rE_\cZ(y)$, then there exists a sequence of elements $a_k \in A$ such that
$$\sum_k a_k a_k^* = x \quad\text{and}\quad \sum_k a_k^* a_k = y \; .$$
\end{lem}
\begin{proof}
The result follows from a maximality argument and the equality $\rE_\cZ(a^* a) = \rE_\cZ(a a^*)$ for all $a \in A$.
\end{proof}

\begin{rem}
By Proposition \ref{prop.quasi-normalizer}, for the same groups $\Gamma$ as in Theorem \ref{thm.stable-strong-solidity}, the group von Neumann algebra $M = \rL(\Gamma)$ has the following property: for every diffuse abelian von Neumann subalgebra $A \subset M$, the quasi-normalizer $\QN_M(A)\dpr$ is amenable. However, we leave open the question whether $\QN_M(A)\dpr$ is amenable for every diffuse amenable subalgebra $A \subset M$. Note here that in specific case of the free group factors $M = \rL(\F_n)$ and using free entropy dimension, it was proved in \cite{Vo95} that for a diffuse abelian $A \subset M$, the $*$-algebra $\QN_M(A)$ cannot be dense in $M$. In \cite{Ha15}, this result was generalized to arbitrary diffuse amenable subalgebras $A \subset M$.
\end{rem}

\subsection{Relative stable strong solidity}

In Proposition \ref{prop.WC}, we showed how to adapt the weak compactness of \cite{OP07} so as to cover the stable normalizer $\sN_M(A)\dpr$ rather than the normalizer $\cN_M(A)\dpr$. In exactly the same way, the methods of \cite[Section 5.1]{PV11} can be extended to the stable normalizer. As a consequence, one obtains the following improvement of \cite[Theorem 1.6]{PV11}.

\begin{thm}\label{thm.relative-stable-strong-solidity}
Let $\Gamma$ be a countable group with the CMAP that admits a proper $1$-cocycle into an orthogonal representation that is weakly contained in the regular representation. Assume that $\Gamma \actson (B,\tau)$ is any trace preserving action on the tracial von Neumann algebra $(B,\tau)$. Put $M = B \rtimes \Gamma$ and let $A \subset M$ be a von Neumann subalgebra that is amenable relative to $B$. Then at least one of the following statements holds.
\begin{itemize}
\item [$(1)$] $A \prec_M B$.
\item [$(2)$] $\sN_M(A)\dpr$ remains amenable relative to $B$.
\end{itemize}
\end{thm}

As a consequence of Theorem \ref{thm.relative-stable-strong-solidity}, also the results of \cite{Io12,Va13} on the normalizer of subalgebras $A$ of amalgamated free products $M = M_1 *_B M_2$ generalize to the stable normalizer $\sN_M(A)\dpr$.

On the other hand, it is not clear to us whether Proposition \ref{prop.WC} and Theorems \ref{thm.stable-strong-solidity} and \ref{thm.relative-stable-strong-solidity} remain valid if we replace CMAP by weak amenability because so far, we were unable to extend the methods of \cite[Section 4]{Oz10} to the stable normalizer.

\section{Proof of the Main Theorem}

The following general lemma will be the key to deduce strong solidity results for type $\rm{III}$ factors from structural results of their continuous core.

\begin{lem}\label{lem-modular}
Let $Q \subset M$ be any inclusion of $\sigma$-finite von Neumann algebras with faithful normal conditional expectation $\rE_Q : M \to Q$. Let $\varphi \in M_\ast$ be any faithful state such that $\varphi \circ \rE_Q = \varphi$. Then for any $u \in \mathcal N_M(Q)$ and any $t \in \R$, we have $u^*\sigma_t^\varphi(u) \in Q \vee (Q' \cap M)$.
\end{lem}
\begin{proof}
Put $\varphi_u := \varphi \circ \Ad(u)$ and note that $Q$ is globally invariant under the modular automorphism groups $\sigma^\varphi$ and $\sigma^{\varphi_u}$. Put $\psi := \varphi |_Q$ and $\psi_u := \varphi_u |_Q$. Observe that $\sigma_t^{\psi_u}(x) = \sigma_t^{\varphi_u}(x)$ and $\sigma_t^\psi (x) = \sigma_t^\varphi (x)$ for every $t \in \R$ and every $x \in Q$. By \cite[Th\'eor\`eme 1.2.1]{Co72}, there exists a $\sigma$-strongly continuous map $\R \to \mathcal U(Q): t \mapsto w_t$ such that $\sigma_t^{\psi_u} = \Ad(w_t) \circ \sigma_t^\psi$ for every $t \in \R$.

By \cite[Lemme 1.2.3(c)]{Co72}, we have $\sigma_t^{\varphi_u} = \Ad(u^* \sigma_t^\varphi(u)) \circ \sigma_t^\varphi$ for every $t \in \R$. Since
$$w_t \, \sigma_t^\varphi(x) \, w_t^* = \sigma_t^{\varphi_u}(x) = (u^* \sigma_t^\varphi(u)) \, \sigma_t^\varphi(x) \, (u^* \sigma_t^\varphi(u))^*$$
for every $x \in Q$, it follows that $w_t^* \, (u^* \sigma_t^\varphi(u)) \in Q' \cap M$. Therefore, we have $u^* \sigma_t^\varphi(u) \in Q \vee (Q' \cap M)$ for every $t \in \R$.
\end{proof}

\begin{proof}[Proof of the main theorem]
First, observe that every free Araki--Woods factor is contained with expectation in a free Araki--Woods factor of type $\rm{III}_1$. Indeed, for any orthogonal representation $U : \R \curvearrowright H_\R$, we have the following inclusion with expectation
\[\Gamma(H_\R,U)\dpr \subset \Gamma(H_\R \oplus \rL^2_\R(\R),U \oplus \lambda)\dpr,\]
where $\lambda : \R \curvearrowright \rL^2_\R(\R)$ is the regular representation and $\Gamma(H_\R \oplus \rL^2_\R(\R),U \oplus \lambda)\dpr$ is a free Araki--Woods factor of type ${\rm III_1}$. Since strong solidity is preserved under taking diffuse subalgebras with expectation, we may assume that $M := \Gamma(H_\R,U)\dpr$ is a free Araki--Woods factor of type $\rm{III}_1$.

Let $Q \subset M$ be any diffuse amenable von Neumann subalgebra with expectation. We want to prove that $P := \cN_M(Q)\dpr$ is amenable. Fix a faithful state $\psi \in M_\ast$ such that $Q$ is globally invariant under the modular automorphism group $\sigma^\psi$.
Observe that $Q' \cap M$, $\mathcal N_M(Q)\dpr$ and $\mathcal N_M(Q \vee (Q' \cap M))\dpr$ are all globally invariant under the modular automorphism group $\sigma^\psi$. Since $M$ is solid by \cite{Oz03, VV05} (see also \cite[Theorem A]{HR14}), $Q' \cap M$ is amenable and so is $\mathcal Q = Q \vee (Q' \cap M)$. Observe that $\mathcal Q' \cap M = \mathcal Z(\mathcal Q)$. Since the inclusion $\mathcal N_M(Q)\dpr \subset \mathcal N_M(\mathcal Q)\dpr$ is with expectation, it suffices to prove that $\cN_M(\mathcal Q)\dpr$ is amenable. Therefore, without loss of generality, we may further assume that $\mathcal Q = Q$, that is, $Q' \cap M = \mathcal \cZ(Q)$.

Now, assume by contradiction that $P = \cN_M(Q)\dpr$ is not amenable. Then there exists a nonzero central projection $z \in \cZ(P)$ such that $Pz$ has no amenable direct summand. Since $\cZ(P) \subset Q' \cap M = \cZ(Q)$, we have $z \in \cZ(Q)$. Then we have $Qz \subset zMz$, $(Qz)' \cap zMz = \cZ(Qz)$ and $\cN_{zMz}(Qz)\dpr = Pz$ has no amenable direct summand. Since $zMz \cong M$, we may replace $Q \subset M$ by $Qz \subset zMz$ and assume without loss of generality that $P =  \cN_M(Q)\dpr$ has no amenable direct summand.

{\bf Claim.}
We have $\core_\psi(P) \subset \mathcal N_{\core_\psi(M)}(\core_\psi(Q))\dpr$.

Since $\lambda_\psi(t) \in \mathcal U(\core_\psi(Q))$, we have $\lambda_\psi(t) \in \mathcal N_{\core_\psi(M)}(\core_\psi(Q))$ for every $t \in \R$. Let now $u \in \cN_M(Q)$. Then we have $\pi_\psi(u) \pi_\psi(Q) \pi_\psi(u)^* = \pi_\psi(Q) \subset \core_\psi(Q)$. Moreover since $Q' \cap M = \mathcal Z(Q)$, Lemma \ref{lem-modular} shows that we have $u\sigma^\psi_t(u^*) \in \mathcal U(Q)$  for all $t \in \R$. Therefore, we have $\pi_\psi(u)\lambda_\psi(t) \pi_\psi(u)^* = \pi_\psi(u\sigma^\psi_t(u^*))\lambda_\psi(t) \in \pi_\psi(Q)\lambda_\psi(t) \subset \core_\psi(Q)$ for all $t \in \R$. Altogether, we have that $\pi_\psi(u)\core_\psi(Q)\pi_\psi(u)^* \subset \core_\psi(Q)$. Replacing $u \in \cN_M(Q)$ by $u^* \in  \cN_M(Q)$, we obtain $\pi_\psi(u)^*\core_\psi(Q)\pi_\psi(u) \subset \core_\psi(Q)$ and hence $\pi_\psi(u)\core_\psi(Q)\pi_\psi(u)^* = \core_\psi(Q)$. Since
$$\core_\psi(P) = \bigvee \left\{ \pi_\psi(u), \lambda_\psi(t) \mid u \in \mathcal N_M(Q), t \in \R\right\},$$ we finally have $\core_\psi(P) \subset \mathcal N_{\core_\psi(M)}(\core_\psi(Q))\dpr$. This proves the claim.

Since $P$ has no amenable direct summand, $\core_\psi(P)$ has no amenable direct summand either by \cite[Proposition 2.8]{BHR12}. The above claim further implies that $\cN_{\core_\psi(M)}(\core_\psi(Q))\dpr$ has no amenable direct summand.
Denote by $\varphi = \varphi_U$ the free quasi-free state on $M$ and put $M_0 := \core_\varphi(M)$, $Q_0 := \Pi_{\varphi,\psi}(\core_\psi(Q))$ and $P_0 := \Pi_{\varphi,\psi}(\cN_{\core_\psi(M)}(\core_\psi(Q))\dpr) = \cN_{M_0}(Q_0)\dpr$.

Since $Q$ is diffuse, \cite[Lemma 2.5]{HU15} and the paragraph following \cite[Theorem 2.2]{HU15} imply that $pQ_0p \nprec_{M_0} \rL_\varphi(\R)$ for any nonzero finite trace projection $p \in \Pi_{\varphi,\psi}(\rL_\psi(\R))$. Take such a projection $p \in \Pi_{\varphi,\psi}(\rL_\psi(\R)) \subset Q_0$, so that
\begin{equation}\label{nonembed}
pQ_0p \nprec_{M_0} \rL_\varphi(\R).
\end{equation}
Since $M$ is a type $\rm{III}_1$ factor, $M_0$ is a type ${\rm II_\infty}$ factor and hence there exists a unitary $u \in \mathcal U(M_0)$ such that $upu^* \in \rL_\varphi(\R)$. Therefore, up to conjugating $Q_0$ (and $P_0$) by a unitary in $\mathcal U(M_0)$, we may assume that $p \in \rL_\varphi(\R)$. We still have \eqref{nonembed}.

By \cite[Theorem A]{HR10}, $M$ has the CMAP and so does $M_0$ by \cite[Lemma 4.6 and Theorem 4.9]{AD93}. Following \cite[Section 4.1]{HR10}, we know that $N := pM_0p$ admits an $s$-malleable deformation in the sense of Popa such that the $N$-$N$-bimodule $\rL^2(\widetilde N \ominus N)$ is weakly contained in the coarse $N$-$N$-bimodule $\rL^2(N) \otimes \rL^2(N)$. Then \cite[Theorem 4.3]{HR10} and \eqref{nonembed} imply that this deformation does not converge uniformly on $\cU(pQ_0p)$. Put $P_1  := \sN_{pM_0p}(pQ_0p)\dpr$ and observe that $pP_0p \subset P_1$. Proposition \ref{prop.malleable} implies that there exists a nonzero projection $q \in \cZ(P_1)$ such that the $qNq$-$N$-bimodule $q\rL^2(\widetilde N \ominus N)$ is left $P_1q$-amenable. Therefore, the coarse $qNq$-$N$-bimodule is left $P_1q$-amenable by \cite[Corollary 2.5]{PV11} which further implies that $P_1q$ is amenable by \cite[Proposition 2.4]{PV11}. This however contradicts the fact that $P_0$ has no amenable direct summand.
\end{proof}

\section{Further remarks on stable strong solidity}

In this section, we clarify the relationship between stable strong solidity and strong solidity. Lemma \ref{lem.prop-PN}(3) shows that if the infinite amplification $M \ovt \rB(\ell^2(\N))$ of a diffuse $\sigma$-finite von Neumann algebra $M$ is strongly solid then $M$ is stably strongly solid. We will show that the converse is also true. First, we prove the following useful lemma.

\begin{lem}
\label{semifinite}
Let $(M, \Tr)$ be any semifinite von Neumann algebra endowed with a fns trace.  Assume that $N \subset M$ is a subalgebra with expectation $\rE_N : M \to N$ such that $N' \cap M \subset N$.

Then $\Tr$ is semifinite on $N$ and $\Tr \circ \rE_N = \Tr$.
\end{lem}

\begin{proof}
Since $M$ is semifinite and $N \subset M$ is with expectation, $N$ is a semifinite von Neumann algebra. Denote by $\Tr_N$ a fns trace on $N$ and define a fns weight $\varphi$ on $M$ by the formula $\varphi := \Tr_N \circ \rE_N$. Denote by $h$ the unique positive self-adjoint operator affiliated with $M$ such that $\varphi = \Tr(h \, \cdot)$ (see \cite[Theorem 5.12]{PT72}).

Note that $N$ centralizes both weights $\Tr$ and $\varphi$. Then for all $t \in \R$, the Radon--Nikodym derivative $({\rm D}\varphi:{\rm D}\Tr)_t = h^{{\rm i}t}$ commutes with $N$. This shows that $h^{{\rm i}t} \in N' \cap M = \cZ(N)$ for all $t \in \R$ and hence $h$ is affiliated with the center $\cZ(N)$ of $N$. Since $\vphi$ is semifinite on $N$ and $h$ is a nonsingular operator affiliated with $\cZ(N)$, also $\Tr$ is semifinite on $N$. Since $\vphi \circ \rE_N = \vphi$, also $\Tr \circ \rE_N = \Tr$.
\end{proof}

\begin{cor}
Let $M$ be any diffuse $\sigma$-finite von Neumann algebra. Then the following facts are true:
\begin{itemize}
\item [$(1)$] $M$ is solid if and only if its infinite amplification $M \ovt \rB(\ell^2(\N))$ is solid.

\item [$(2)$] $M$ is stably strongly solid if and only if its infinite amplification is strongly solid.
\end{itemize}
\end{cor}
\begin{proof}
First, observe that a direct sum of von Neumann algebras is solid (resp.\ (stably) strongly solid) if and only if each direct summand is solid (resp.\ (stably) strongly solid). Next, any solid von Neumann algebra with diffuse center is amenable. Also, the notions of stable strong solidity and strong solidity coincide for properly infinite von Neumann algebras by Lemma \ref{lem.prop-PN}(2, 3). Therefore, we only need to consider the case where $M$ is a ${\rm II_1}$ factor. Denote by $M_\infty := M \ovt \rB(\ell^2(\N))$ the infinite amplification of $M$ equipped with the canonical fns trace $\Tr := \tau \otimes \Tr_{\rB(\ell^2(\N))}$.

(1) If $M_\infty$ is solid then $M$ is solid as well since $M \subset M_\infty$ is with expectation and solidity is preserved under taking diffuse subalgebras with expectation. Assume now that $M$ is solid and take a diffuse subalgebra $N \subset M_\infty$ with expectation. Take a diffuse abelian subalgebra $A \subset N$ with expectation. To prove that $N' \cap M_\infty$ is amenable, it is sufficient to prove that $A' \cap M_\infty$ is amenable since $N' \cap M_\infty \subset A' \cap M_\infty$ is with expectation.

Since $A$ is abelian, we have $A \subset A' \cap M_\infty:= Q$ and hence $Q' \cap M_\infty \subset Q$.  Then Lemma \ref{semifinite} implies that the semifinite trace $\Tr$ on $M_\infty$ remains semifinite on $Q$.  Since $Q$ is diffuse, $\Tr|_Q$ is semifinite and $\Tr(1) = +\infty$, we may take a sequence of pairwise orthogonal projections $p_n \in Q$ such that $\Tr(p_n) = 1$ for all $n$ and $\sum_n p_n = 1$. Then for all $n$, we have that $p_nQp_n = (Ap_n)' \cap p_nM_\infty p_n$. Since $p_nM_\infty p_n \cong M$ is solid, we have that $p_nQp_n$ is amenable for all $n$ and we conclude that $Q$ is amenable.

(2) Lemma \ref{lem.prop-PN}(3) shows the ``if'' part. Let us show the ``only if'' part. Assume that $M$ is stably strongly solid and take a diffuse amenable subalgebra $Q \subset M_\infty$ with expectation.

By (1), $M_\infty$ is solid and hence $\cQ := Q \vee (Q' \cap M_\infty)$ is also diffuse amenable and $\cN_{M_\infty}(Q)\dpr$ is contained with expectation inside $\cN_{M_\infty}(\cQ)\dpr$. Hence replacing $Q$ by $\cQ$ if necessary, we may assume that $Q' \cap M_\infty \subset Q$.
Then Lemma \ref{semifinite} implies that the semifinite trace $\Tr$ on $M_\infty$ remains semifinite on $Q$.

Take a sequence of pairwise orthogonal projections $p_n \in Q$ such that $\Tr(p_n) = 1$ for all $n$ and $\sum_n p_n = 1$. Since $p_nM_\infty p_n \cong M$ is stably strongly solid, we have that $\sN_{p_nM_\infty p_n}(p_nQp_n)\dpr$ is amenable for all $n$. Hence $p_n (\cN_{M_\infty}(Q)\dpr) p_n \subset p_n (\sN_{M_\infty}(Q)\dpr) p_n = \sN_{p_nM_\infty p_n}(p_nQp_n)\dpr$ is amenable for all $n$ and we conclude that $\cN_{M_\infty}(Q)\dpr$ is amenable as well.
\end{proof}


\bibliographystyle{plain}

\end{document}